\documentclass[12pt,reqno]{amsart} % declare the document type
\usepackage{amsmath, amsfonts, amssymb, amsthm, amscd, amscd, amsbsy} 
\usepackage{fancyhdr} % customize the page layout
\usepackage{graphicx} % manage graphs
\usepackage{geometry} % handle page layout
\usepackage[]{hyperref} % manage reference
\usepackage{verbatim} % handle comment lines
\usepackage{fullpage} % handle the margins
\usepackage{setspace} % handle spacing
\usepackage{multicol}
\usepackage{enumitem}
\usepackage{parskip}
\usepackage{cite}
\usepackage{amsrefs}
\usepackage{extarrows}% http://ctan.org/pkg/extarrows
\usepackage{xcolor}
\usepackage[toc,page]{appendix}

\newtheorem{theorem}{Theorem}[section]
\newtheorem{lemma}[theorem]{Lemma}
\newtheorem{proposition}[theorem]{Proposition}

\theoremstyle{definition}

\newtheorem{remark}[theorem]{Remark}
\numberwithin{equation}{section}

\DeclareMathOperator\sech{sech}

\newcommand{\rn}{\mathbb{R}^n}
\newcommand{\hn}{\mathbb{H}^n}

\newcommand{\h}{\mathbb{H}^3}

\newcommand{\parder}[2]{{\frac{\partial {#1}}{\partial {#2}}}}

\newcommand{\pb}{\partial B_1}

\begin{document}
\title{Conformally  invariant equations with negative critical exponents on the three dimensional hyperbolic space} % title of the document
\author{Debdip Ganguly, Jungang Li, Guozhen Lu, Jianxiong Wang}

\address{Debdip Ganguly, Theoretical Statistics and Mathematics Unit, Indian Statistical Institute, Delhi Centre, New Delhi, Delhi, India.}
\email{debdip@isid.ac.in}
\address{Jungang Li: Department of Mathematics\\University of Science and Technology of China\\Heifei, Anhui, China.}
\email{jungangli@ustc.edu.cn}
\address{Guozhen Lu: Department of Mathematics\\University of Connecticut\\Storrs, CT 06269, USA.}
\email{guozhen.lu@uconn.edu}
\address{Jianxiong Wang: Department of Mathematics\\Rutgers University\\Piscatawat, NJ 08854, USA.}
\email{jianxiong.wang@rutgers.edu}

\begin{abstract}
We establish a symmetry result for positive entire solutions with a prescribed growth rate to the following fourth order equation on the 3-dimensional hyperbolic space $\mathbb{H}^3$:  
\[
P_2 u = - u^{-7},
\]
where $P_2$ denotes the fourth-order Paneitz operator. We prove that any positive solution $u$ on $\mathbb{H}^3$ exhibiting exponential growth at infinity must, up to hyperbolic isometries, be radial and strictly decreasing with respect to some point $P \in \mathbb{H}^3$. 

Fourth order equations with negative critical growth on 3-dimensional Euclidean space $\mathbb{R}^3$ has been studied by Choi and Xu in \cite{CX09 }, and subsequently  by  McKenna and Reichel \cite{MR03} and Xu \cite{Xu05}. Unlike the Euclidean case, the behavior of the Green's function of $P_2$ is substantially different, which prevents us from using the moving plane (sphere) method directly.

\end{abstract}

\maketitle

\section{Introduction}
The classification of solutions to nonlinear elliptic equations plays a central role in the study of partial differential equations and geometric analysis, particularly in problems involving curvature prescription and variational structures. In the famous work of Gidas--Ni--Nirenberg~\cite{GNN79}, the authors considered the boundary value problem
\begin{equation*}
    \begin{cases}
        -\Delta u = f(u) & \text{in } B_R(0)\subset \mathbb{R}^n,\\[4pt]
        u = 0 & \text{on } \partial B_R(0),
    \end{cases}
\end{equation*}
with $f \in C^1$. They proved that any positive solution $u \in C^2(\overline{B_R(0)})$ is radially symmetric with respect to the center of the ball and strictly decreasing along the radial direction. Their argument is based on the method of moving planes, a symmetry technique originally introduced by Alexandrov~\cite{Alex62} and later refined by Serrin~\cite{Seri71}, as well as 
Berestycki and Nirenberg \cite{BN91}. In a subsequent work, Gidas--Ni--Nirenberg~\cite{GNN81} established symmetry results for solutions of the equation $-\Delta u = f(u)$ in the punctured space $\mathbb{R}^n\setminus\{0,\infty\}$, $n \ge 3$, where the solution develops isolated singularities at both the origin and at infinity, satisfying
\begin{align}
  u(x) \to +\infty \quad &\text{as } x\to 0, \nonumber\\
  |x|^{n-2}u(x) \to +\infty \quad &\text{as } |x|\to\infty.
\end{align}
They showed that every such solution is radially symmetric about the origin and strictly decreasing. Further developments were obtained by Caffarelli--Gidas--Spruck~\cite{CGS89}, who studied the same equation in a punctured ball under critical growth conditions. More precisely, they assumed that the nonlinearity $f$ is locally Lipschitz, nondecreasing, satisfies $f(0)=0$, and for sufficiently large $t$, the quantity $t^{-\frac{n+2}{n-2}}f(t)$ is nonincreasing and $f(t)\ge c t^p$ for some $p \ge \frac{n}{n-2}$. Under these assumptions, they proved that all positive solutions remain radially symmetric and strictly decreasing.

The method of moving planes has also been adapted to study higher-order 
elliptic equations in \(\mathbb{R}^n\). The main difficulty in this context 
stems from the fact that the classical moving plane technique relies heavily 
on the maximum principle, which is generally not valid for higher-order 
operators. While Hopf-type lemmas can sometimes provide a substitute, such 
results are only available in special situations (see e.g. 
\cite{BGW08}). To circumvent this limitation, Chen, Li, and Ou \cites{CLO06, ChenLiOu2} introduced a 
powerful adaptation of the moving plane method formulated for integral 
equations. (see also \cite{ChenWu-ANS, ChenHuMa-ANS}.) More specifically, their moving plane technique can be 
effectively applied to the following integral equation, yielding 
radial symmetry and monotonicity properties for its positive regular (i.e. $L^{\frac{2n}{n-\alpha}}_{loc}$) solutions.
\begin{equation}\label{eq1}
  u(x)=\int_{\rn}\frac{u(y)^{\frac{n+\alpha}{n-\alpha}}}{|x-y|^{n-\alpha}},\; \forall x\in\rn.
\end{equation}
Consequently, such solutions must take the explicit form
\begin{equation}\label{soln}
    u(x) = c \left( \frac{t}{t^2 + |x-x^0|^2} \right)^{\frac{n-\alpha}{2}}.
\end{equation}
Moreover, their work provides a complete classification of all critical points 
of the functional associated with the Hardy-Littlewood-Sobolev inequalities of 
order \(\alpha\), whose sharp constants were previously determined by Lieb 
\cite{Lieb83}. This unified approach not only extends the symmetry results to 
fractional powers \(\alpha \in (0,n)\) but also connects the problem to sharp 
functional inequalities in analysis.

The moving plane method on $\mathbb{H}^n$ was first seen in the work of Kumaresan-Prajapat  \cites{K-P1,K-P2}, 
where they established the analogous result of Gidas-Ni-Nireneberg type, as well as solving a Serrin type overdetermined problem
on $\mathbb{H}^n$. The concept of the moving plane method in \cites{K-P1,K-P2} was further developed in the work of Almeida-Ge \cite{ADG2} and Almeida-Damascelli-Ge \cite{ADG1},
where they took advantage of the foliation structure of $\mathbb{H}^n$.  Such a moving plane method was used earlier in the study of the constant mean curvature surface problem by 
Korevaar, Kusner, Meeks and Solomon \cites{KK,KKMS}. Moreover, as long as there is a certain group of symmetries 
(either isometries or conformal maps) for the underlying space which also interacts nicely with the given equation, 
then one can develop these techniques similar to the moving plane method in the Euclidean space. These have been demonstrated nicely by Birindelli and Mazzeo \cite{BirindelliMazzeo}. See also the expository article \cite{Chow}. 
Recently, an existence and symmetry result of the higher order Brezis-Nirenberg problem on $\hn$ was established by Li, Lu and Yang \cite{LLY}. Among other results, they proved the existence of positive $W^{k,2}$ solution of the following equation 
\begin{equation}\label{BZ}
  P_k u - \lambda u = u^{\frac{n+2k}{n-2k}}, \ n > 2k,
\end{equation}
where $P_k$ denotes the $2k$-th order GJMS operator when $k\ge 2$ and $\lambda<\prod\limits_{i=1}^{k}\frac{\left(2i-1\right)^2}{4}$. Moreover, they show that such solutions must be symmetric with respect to some point $P \in \mathbb{H}^n$.
To achieve this result, they made use of the Helgason-Fourier analysis \cites{H-F1, H-F2} and developed a moving plane method for integral equations on $\mathbb{H}^n$. Li, Lu and Wang \cite{LLW, LLW2} later studied the symmetry of solutions to a more general class of higher order PDE involving GJMS operators and fractional order equations. More recently, Li, Lu and Wang \cite{LLW25} developed a moving sphere method on hyperbolic spaces and obtained classification results to higher order equatons.
 We mention that 
for $k=1$, Mancini and Sandeep \cite{MS} showed earlier  that the entire solutions exist and are symmetric either when  $n\geq4$, $p=\frac{2n}{n-2}$, $0<\lambda\leq\frac{1}{4}$ or when $n\geq3$, $1<p<\frac{2n}{n-2}$ and $\lambda\leq\frac{1}{4}$.   More recently, Lu and Tao \cite{LuTao} establish the existence and symmetry  of solutions to \eqref{BZ} when $k\ge 2$ and $\lambda=\prod\limits_{i=1}^{k}\frac{\left(2i-1\right)^2}{4}$ by using the concentration-compactness principle. 
 
In the present paper, we turn our attention to the cases when the dimension $n< 2k$. One observes a conformally invariant equation with negative critical exponents on hyperbolic spaces:
\begin{equation}\label{eq:main-H}
  P_k u \;=\; -u^{\frac{n+2k}{\,n-2k\,}}.
\end{equation} 
On $\hn$,  the GJMS operator reads as 
\begin{equation*}
    P_k=P_1(P_1+2)\cdots(P_1+k(k-1)),\quad k\in\mathbb{N},
\end{equation*}
where $P_1 = - \Delta_{\mathbb{H}} - \frac{n(n-2)}{4}$ is the conformal Laplacian on $\hn$.
By letting $v = \left( \frac{2}{1 - |x|^2} \right)^{\frac{n-2k}{2}} u$, one gets the following equivalent biharmonic equation on $B_1(0) \subset \mathbb{R}^n$:
\begin{equation}\label{eq:biLap}
      (-\Delta)^k v  + v^{-\frac{n+2k}{n-2k}} = 0.
\end{equation}
When $n = k = 1$, the above equation reduces to a second order ordinary differential equation with negative exponent.  
 When $n = 3$ and $k = 2$, \eqref{eq:biLap} on the whole space was studied by Choi and Xu \cite{CX09} as well as by McKenna and Reichel \cite{MR03} . The main result in Choi and Xu is that if $v$ has exact linear growth at infinity in the sense that $\lim\limits_{|x|\to+\infty} \frac{v(x)}{|x|}=\alpha>0$, then $v$ solves the integral equation
\begin{equation*}
    v(x)=\int_{\mathbb{R}^3}|x-y|\cdot v^{-7} (y) \;dy.
\end{equation*}
It admits (up to translation) only one kind of entire solutions given by
\begin{equation*}
    v(x)=\alpha\left(\frac{1}{\sqrt{15\alpha^8}+|x|^2}\right)^{-\frac{1}{2}}.
\end{equation*}
Moreover, Choi and Xu proved that \eqref{eq:biLap} has no linear growth solution if the exponent is replaced by $q$ satisfying $4 < q < 7$. Thus, \eqref{eq:biLap} in dimension $n = 3$ with $q = 7$ is a very distinguished case. For the triLaplace equation on $\mathbb{R}^5$ for $q=11$, the solutions are studied by Feng and Xu \cite{FengXu2013}. The general $2k$-th order polyharmonic equation
\begin{equation*}
    (-\Delta)^k u + u^{-(4k-1)} = 0 \text{ in } \mathbb{R}^{2k-1}
\end{equation*}
was studied recently by Ng\^o \cite{Ngo18}. For the corresponding integral equation, the classification of positive solutions was obtained by Li \cite{YL04}, using a moving sphere method. Mckenna and Reichel \cite{MR03} proved that there are infinitely many radial solutions with almost quadratic, superlinear growth rates to \eqref{eq:biLap} with $q=7$.
The existence of infinitely many entire solutions with different growth rates for the conformally invariant equation 
$\Delta^2u + u^{-7} = 0$ in $\mathbb{R}^3$ is in striking contrast to the conformally invariant equation 
$\Delta^2 u = u^\frac{n+4}{n-4}$ in $\rn$ with $n \geq 5$ and the second order equation $-\Delta u = u^{\frac{n+2}{n-2}}$ 
in $\rn$ with $n \geq 3$, where there exists a unique one-parameter family of positive entire solutions.

The aforementioned equations are closely connected to the prescribed 
\(Q\)-curvature problem. In the Euclidean 
setting, the classification of positive entire solutions
serves as the fundamental ingredient in the blow-up analysis for the prescribed 
curvature problem on compact Riemannian manifolds. Such classification results 
provide the local models for concentration phenomena and are essential in 
understanding compactness, quantization, and the structure of possible 
singularity profiles. It is then natural to ask the analogous question on complete noncompact manifolds. 
In the present article, we make a preliminary attempt by studying the analogous equation in $\h$, the space form with negative constant curvature.
The Paneitz operator of fourth order and general GJMS operators arise naturally in conformal geometry.
For odd dimensions $n$ and $k\geq 1$, the GJMS-operator $P_k$ gives rise to a curvature quantity $Q_{k}$ 
via the formula  
\begin{equation}\label{GJMS_Q}
    P_k (g)(1) = (-1)^k (\frac{n}{2}- k)Q_{k}(g).
\end{equation}
We remark here that if the dimension is even, \eqref{GJMS_Q} holds only when $n>2k$ for general Riemannian manifolds, and the critical $Q$-curvatures $Q_n$ (for even $n$) can be defined through the non-critical $Q$-curvatures $Q_k\; (2k < n)$ by a limiting process. For odd $n$, the quantities $Q_k$ are defined for all $k \geq 2$ by \eqref{GJMS_Q}.

A key feature of the GJMS operators is their conformal covariance. If the 
metric \(h\) is conformally equivalent to \(g\) by \(h = u^{\frac{4}{n-2k}}g\), 
then the covariance property reads
\[
    P_k^g(u\varphi)
    = u^{-(4k-1)} P_k^h(\varphi),
    \qquad \forall\, \varphi \in C^\infty(M).
\]
Prescribing the \(Q\)-curvature of the metric \(h\), denoted by \(Q_h\), leads 
to the nonlinear equation
\begin{equation}\label{Qeq_2}
    P_k^g(u) = u^{-(4k-1)} P_k^h(1)
\end{equation}
up to a multiplicative constant. In particular, when the ambient space is 
\(\mathbb{H}^3\), the positive prescribed \(Q\)-curvature is a positive constant, and 
\(k=2\), equation~\eqref{Qeq_2} reduces to
\begin{equation}\label{eq:Paneitz}
    P_2 u = - u^{-7},
\end{equation}
which will be the central object of study in the present article.

A fundamental step in the analysis of \eqref{eq:biLap} on $\mathbb{R}^3$ 
is the observation that the equation is equivalent to the nonlinear integral 
formulation
\begin{equation}\label{eq:integral-euclid}
    v(x) = \int_{\mathbb{R}^3} |x-y|\, v^{-7} (y) \, dy.
\end{equation}
This equivalence holds under an appropriate linear-growth condition at infinity 
and is ultimately justified by a Liouville-type theorem. Motivated 
by this analogy, it is natural to conjecture that any sufficiently regular entire, i.e., $W^{2,2}$ weak solution of the hyperbolic equation \(P_2 u = -u^{-7}\) should also satisfy 
a corresponding Green's representation formula, possibly after applying suitable 
normalization or scaling.

With this viewpoint, we turn to the integral formulation of the problem on 
hyperbolic space:
\begin{equation}\label{eq:green-IE}
    u(x) = \int_{\mathbb{H}^3} G(x,y)\, u^{-7}(y)\, dV_y,
\end{equation}
where \(G(x,y)\) denotes the Green's function of the operator \(P_2\) on 
\(\mathbb{H}^3\). The explicit formula of Green's functions was recently established in 
\cites{LY19, LY22}. For dimensions \(n > 2k\), such integral representations 
can be effectively analyzed using the Helgason--Fourier transform on 
\(\mathbb{H}^n\); see, for instance, \cites{LLY, LLW, LLW25}. 

In the present case, however, the anticipated behavior of solutions is 
\emph{exponential growth}, specifically \(u(x) \sim e^{\frac{\rho(x)}{2}}\) as 
\(\rho(x) \to \infty\) . This growth regime lies outside the admissible class for 
the Helgason--Fourier transform, and therefore the Helgason-Fourier analysis 
techniques do not apply directly. As a consequence, the correspondence between the differential formulation and its integral representation breaks down in the hyperbolic setting, revealing a 
phenomenon that has no analogue in the Euclidean case. This observation indicates the emergence of a Liouville-type phenomenon on 
\(\mathbb{H}^3\) within this regime, a topic that will be further explored in 
Section~\ref{apx}. The behavior of Green's function on $\mathbb{H}^3$ differs substantially from the Euclidean 
one, which exhibits 
linear growth at infinity. The decay of the Green's kernel in $\mathbb{H}^3$ introduces additional 
analytic subtleties.

To proceed, we reformulate the original problem \eqref{eq:Paneitz} on 
\(\mathbb{H}^3\) into an equivalent boundary value problem on the Euclidean 
unit ball. Rather than working directly with the critical power 
nonlinearity, we first consider a more general semilinear biharmonic problem 
posed on \(B_1(0) \subset \mathbb{R}^3\):
\begin{equation}\label{eq:biharmonic}
\begin{cases}
 \Delta^2 v(x) = - f(v), & x\in B_1(0),\\
 v(x) = a, & x\in \partial B_1(0),\\
 \partial_\nu v(x) = b, & x\in \partial B_1(0), 
\end{cases}
\end{equation}
with boundary constants $a, b>0$, where $f: (0, \infty) \to (0, \infty)$ is a continuous, non-increasing function (this $f$ is different from the one in \eqref{eq:green-IE}).
Using Green's identity for the bi-Laplacian (Dirichlet data), one obtains the integral representation
\begin{equation}
v(x)
= \int_{\partial B_1} \Bigl( v\,\partial_\nu (\Delta G) - (\partial_\nu v)\,\Delta G \Bigr) d\sigma_y 
- \int_{B_1} G(x,y)f(v)\,dy.
\end{equation}
Here $G$ is the Green's function for $\Delta^2$ on $B_1(0) \subset \mathbb{R}^3$ with Dirichlet boundary conditions. 

\begin{theorem}\label{integral_rep}
Let \(v\) be a positive $W^{2,2}$ weak solution  of \eqref{eq:biharmonic}, satisfying $\inf |v| \geq \epsilon > 0$ for some $\epsilon$. Then \(v\) admits 
the following integral representation:
\begin{equation}\label{eq:integral_rep}
    v(x)
    = C_1 \;+\; C_2\, (|x|^2 - 1)
      \;-\; \int_{B_1} G(x,y)\, f(v(y)) \, dy,
\end{equation}
where \(B_1\) denotes the unit ball in \(\mathbb{R}^3\) centered at the origin,
and the constants are given by 
\[
    C_1 = 3\sqrt{\pi}\, a, 
    \qquad C_2 = \frac{3\sqrt{\pi}}{2}\, b.
\]
\end{theorem}

After establishing detailed estimates and precise analysis of the Green's 
function, we are now in a position to apply the moving plane method to the 
above equation. This refined analysis allows us to extend 
the classical symmetry argument to the present setting and ultimately 
derive the following result.

\begin{theorem}\label{thm:main}
The solution \(v\) in Theorem \ref{integral_rep} must be radially symmetric with respect to the 
origin. Moreover, \(u\) is strictly increasing as a function of 
\(r = |x|\).
\end{theorem}

\begin{remark}
  By standard regularity theory, any $W^{2,2}$ solution of \eqref{eq:biharmonic} is $C^{3,\gamma}$ for some $0 < \gamma < 1.$ Moreover, if the nonlinearity $f$ is smooth, $v$ is actually the classical solution. Moreover, if the nonlinearity $f (v)$ is replaced by a nonautonomous function $f(|x| , v)$, then from the proof of Theorem \ref{thm:main}, the same conclusion holds if $f$ is continuous, non-increasing in both variables. 
\end{remark}

Let \(B_1(0)\) denote the unit ball in the Euclidean space \(\mathbb{R}^n\).
We recall that the hyperbolic space \(\mathbb{H}^n\) can be represented as the 
Poincar\'e ball model, where the Euclidean ball \(B_1(0)\) is equipped with the 
Poincar\'e metric.  Define 
\[
    v(y) = \left( \frac{2}{1 - |x|^2} \right)^{-1/2} u(x).
\]
It is well-known that we have the identity
\[
 P_k u
  = \left( \frac{2}{1 - |x|^2} \right)^{-(\frac{n}{2}+k)}
    (-\Delta)^k \left( 
        \left( \frac{2}{1 - |x|^2} \right)^{\frac{n}{2}-k} u
    \right).
\]
Specializing to the case \(n=3\) and \(k=2\), the equation \ref{Qeq_2} reduces to the equation
\[
    \Delta^2 v(x) = - \left( \frac{2}{1 - |x|^2} \right)^{\frac{7}{2}} f \left( \left( \frac{2}{1 - |x|^2} \right)^{\frac{1}{2}} v \right).
\]
When $f(t) = t^{-7}$, the above equation becomes $\Delta^2 v = - v^{-7}$. Denote the hyperbolic geodesic distance from $x$ to $0$ as $\rho(x , 0)$, then it is well known that $\rho (x , 0) = \ln \frac{1 + |x|}{1 - |x|}$. Thus, we have the relation:
$$
  v(x) = \left( \frac{2}{1 - |x|^2} \right)^{-\frac{1}{2}} u = (2 \cosh^2 \frac{\rho(x , 0)}{2})^{-\frac{1}{2}} u.
$$
This implies that $\lim\limits_{|x| \to 1^-} v(x) > 0$ if and only if $u$ has exponential growth $e^{\rho(x , x_0) / 2}$ for some $x_0 \in \mathbb{H}^n$. Moreover, the asymptotic behavior of $\frac{\partial v}{\partial \nu}$ near the boundary is determined by the coefficient of exponential growth of $u$. 
With the help of Theorem \ref{thm:main}, we have the following result.

\begin{theorem}\label{thm1}
    Let \(u\) be a positive entire solution of \eqref{eq:Paneitz} in 
    \(\mathbb{H}^{3}\). Assume that
    \[
        \lim_{\rho(x,x_0)\to\infty} \frac{u(x)}{e^{\rho(x,x_0)/2}} = \alpha
    \]
    for some constant \(\alpha > 0\). Then \(u\) must be symmetric with respect to some point $x_0 \in \mathbb{H}^3$. Moreover, $u$ is strictly increasing with respect to the hyperbolic geodesic distance.
\end{theorem}

The structure of the paper is as follows. In Section 2, we introduce the  geometric and analytic preliminaries necessary for our analysis. In Section 3, we derive the integral representation formula for solutions of the biharmonic equation with Dirichlet boundary conditions on the unit ball. Section 4 is devoted to the application of the moving plane method to establish the radial symmetry and monotonicity of solutions. Finally, in Section 5, we discuss the related integral equation and the Liouville-type phenomenon on hyperbolic space.

\section{Preliminaries: Geometric and functional Analytic tools}

In this section, we present an overview of the geometric framework in which our main problems are formulated, together with some straightforward results that will be used repeatedly in the sequel. 
\medskip
\subsection{Models of hyperbolic spaces}
The hyperbolic $n-$space $\hn$ $(n\geq 2)$ is a complete simply connected Riemannian manifold with constant sectional curvature $-1$. There are several analytic models of hyperbolic spaces, all of which are equivalent. One of the commonly used models is the Poincar\'e ball model, which is given by the open unit ball $\mathbb{B}^n=\{x=(x_1,\cdots,x_n):x_1^2+\cdots+x_n^2<1\}\in\mathbb{R}^n$ equipped with the Poincar\'e metric
$$ds^2=\frac{4\left(dx_1^2+\cdots+dx_n^2\right)}{\left(1-|x|^2\right)^2}.$$
The distance from $x\in\mathbb{B}^n$ to the origin is $\rho(x)=\log\frac{1+|x|}{1-|x|}$.
The hyperbolic volume element is $dV=\left(\frac{2}{1-|x|^2}\right)^ndx$.
The hyperbolic gradient is $\nabla_{\mathbb{H}^n}=\frac{1-|x|^2}{2}\nabla$ and the associated Laplace-Beltrami operator is given by
$$\Delta_{\mathbb{H}^n}=\frac{1-|x|^2}{4}\left((1-|x|^2)\Delta+2(n-2)\sum_{i=1}^nx_i\frac{\partial}{\partial x_i}\right).$$

\subsection{Foliations of hyperbolic spaces}\label{sec:hyperbolic}

A foliation is an equivalence relation on a manifold, the equivalence classes being connected, injectively  submanifolds, all of the same dimension.
Let $\mathbb{R}^{n,1}=(\mathbb{R}^{n+1},\cdot)$, where $\cdot$ is the Lorentzian inner product defined by $x\cdot y=-x_0y_0+x_1y_1+\cdots+x_ny_n$. The hyperboloid model
of $\hn$ is the submanifold $\{x\in\mathbb{R}^{n,1}: x\cdot x=-1, x_0>0\}$. A particular directional foliation can be obtained by choosing any $x_i$ direction, $i=1,\cdots,n$.
Without loss of generality, we may choose the $x_1$ direction. Denote $\mathbb{R}^{n,1}=\mathbb{R}^{1,1}\times\mathbb{R}^{n-1}$, where $(x_0,x_1)\in \mathbb{R}^{1,1}$. We define
$A_t=\tilde{A}_t\otimes Id_{\mathbb{R}^{n-1}}$, where $\tilde{A}_t$ is the hyperbolic rotation (also called a boost) in $\mathbb{R}^{1,1}$,
$$\tilde{A}_t=\begin{pmatrix}
    \cosh t & \sinh t \\
    \sinh t & \cosh t
\end{pmatrix}.$$
Let $U=\hn\cap\{x_1=0\}$ and $U_t=A_t(U)$, then $\hn$ is foliated by $U_t$ and $\hn=\bigcup_{t\in\rn}U_t$.
The reflection $I$ is an isometry such that $I^2=Id$ and $I$ fix the hypersurface $U$, by $I(x_0,x_1,x_2,\cdots,x_n)=(x_0,-x_1,x_2,\cdots,x_n)$.
Moreover, the reflecton with respect to $U_t$ is defined as $I_t=A_t\circ I\circ A_{-t}$, and $U_t$ is fixed by $I_t$.

\subsection{GJMS operator and its fundamental solution}

GJMS operators were discovered in the work of Graham, Jenne, Mason, and Sparling \cite{GJMS2} based on
the construction of ambient metric by C. Fefferman and Graham \cite{FeffermanGr1, FeffermanGr}.

A differential operator $D$ is conformally covariant of bidegree $(a,b)\in\mathbb{R}^2$ in dimension $n$ if for any
$n$-dimensional Riemannian manifold $(M,g_0)$,
\begin{equation*}
  g_\omega=e^{2\omega}g_0, \omega\in C^\infty(M) \implies D_\omega=e^{-b\omega}D_0\mu(e^{a\omega}),
\end{equation*}
where for any $f\in C^\infty(M), \mu(f)$ is multiplication by $f$. Here, the subscripts indicate the corresponding metric.
It was shown in \cite{GJMS2} and \cite{Bra} that GJMS operators are
conformally covariant differential operators of bidegree $(\frac{n}{2}-k,\frac{n}{2}+k)$.

\medskip 

Recently, Lu and Yang \cites{LY19,LY22} established an explicit closed-form 
expression for the fundamental solution of the operator \(P_k\) on the 
hyperbolic space \(\hn\). More precisely, they proved that
\[
 P_k^{-1}(\rho)
 = \frac{\Gamma\!\left(\frac{n}{2}\right)}
        {2^n \pi^{\frac{n}{2}} \Gamma(k)\Gamma(k+1)}
   \frac{\bigl(\cosh \frac{\rho}{2}\bigr)^{-n}}
        {\bigl(\sinh \frac{\rho}{2}\bigr)^{\,n-2k}}
   F\!\left(k - \frac{n-2}{2}, \, k; \, k+1; \,
   \cosh^{-2}\!\frac{\rho}{2}\right),
\]
where \(F\) denotes the Gaussian hypergeometric function.  

From this explicit representation, it follows immediately that for any fixed 
\(y \in \mathbb{B}^n\), the corresponding Green's function 
\(G(x,y) = P_k^{-1}(\rho)\) is strictly positive and depends only on the 
geodesic distance \(\rho = \rho(x,y)\). In particular, it is a radially 
symmetric and monotonically decreasing function of~\(\rho\).

Moreover, when \(1 \le k < \frac{n}{2}\), the fundamental solution satisfies 
the following upper bound:
\begin{equation}\label{pk_est}
    P_k^{-1}(\rho)
    \le \frac{1}{\gamma_n(2k)}
    \left[
        \left(\frac{1}{2 \sinh \frac{\rho}{2}}\right)^{n-2k}
        - 
        \left(\frac{1}{2 \cosh \frac{\rho}{2}}\right)^{n-2k}
    \right],
    \qquad \rho > 0,
\end{equation}
where \(\gamma_n(2k)\) is the normalization constant associated with the 
Sobolev-type kernel on \(\hn\).

\medskip

\section{Green function properties, integral representation and Proof of Theorem~\ref{integral_rep}}

The Green's function corresponding to the bi-Laplacian on the unit ball 
\(B_1(0)\subset \mathbb{R}^n\), equipped with  Dirichlet boundary 
conditions, admits an explicit representation. In particular, Boggio derived 
a closed-form expression for this function, now known as \emph{Boggio's formula}
\cites{GGS10,Bogg05}, which provides an explicit characterization of the Green's kernel in this setting. Let us denote $B_1:= B_1(0).$

\begin{lemma}
The Green function for the Dirichlet problem for the bi-Laplace equation on the unit ball $B_1\subset\mathbb{R}^3$ is positive and given by 
\begin{align}
    G(x,y) &= C|x-y| \int_1^{\frac{[XY]}{|x-y|}} \frac{z^2-1}{z^2} dz \nonumber\\
    &= C\left([XY]+\frac{|x-y|^2}{[XY]}-2|x-y|\right),
\end{align}
where $[XY]=\sqrt{|x|^2|y|^2-2x\cdot y+1},$ and $C$ is a positive constant given by 
$
C = \dfrac{\Gamma(\frac{5}{2})}{4\pi^{\frac{3}{2}}}=\dfrac{3}{16\sqrt{\pi}}.
$

\end{lemma}

Let us examine the higher order derivative of the above Green's function. 

\begin{lemma}\label{lemma-gr}
Let $G$ denote the Green function of the Dirichlet problem for the bi-Laplace equation on the unit ball 
$ B_1\subset\mathbb{R}^3.$ Then there holds the following : 

\begin{itemize}
\item[$(a)$] For $x, y \in B_1,$ and $x^* = \frac{x}{|x|^2}$ when $x \neq 0,$
\begin{equation}\label{eq-1gr}
    \Delta_yG(x,y)\;=\;C\left(\frac{2|x|^2}{[XY]}-\frac{4}{|x-y|}+\frac{6}{[XY]}-\frac{4|x|^2}{[XY^3]}(y-x)(y-x^*)\right), 
    \quad x \neq y.
\end{equation} 
\medskip 
\item[$(b)$] For $x \in B_1$ and $y \in \partial B_1,$ and $x^* = \frac{x}{|x|^2}$  when $x \neq 0,$
\begin{align*}\label{eq-2gr}
    \frac{\partial}{\partial \nu}\Delta_y G(x,y)\;&= \; C\left(-\frac{2|x|^4+6}{[XY]^3}(y-x^*)\cdot y + \frac{12|x|^4(y-x)(y-x^*)}{|x-y|^5}(y-x^*)\cdot y\right. \\
    &\left.+\;4(1-|x|^2)\frac{(y-x)\cdot y}{|x-y|^3} - \frac{4|x|^2(y-x^*)\cdot y}{[XY]^3}\right).
\end{align*}
\end{itemize}
\end{lemma}

\begin{proof}
The proof follows a direct but tedious computation. One can see for $n =3$,
\begin{align*}
    \Delta_y |x-y| = \frac{n-1}{|x-y|}, \; \Delta_y [XY] = \frac{(n-1)x^2}{[XY]},
\end{align*}
and
\begin{align*}
  \Delta_y \frac{1}{[XY]} &= -\frac{1}{[XY]}\Delta_y [XY] + \frac{2}{[XY]^3}\frac{\sum_{i=1}^n (y_i|x|-x_i)^2}{[XY]^2}\\
  &= (3-n)\frac{|x|^2}{[XY]^3}.
\end{align*}
In particular, we have $\Delta_y \frac{1}{[XY]}=0$ when $n=3$. Secondly, 
\begin{align*}
  \Delta_y \frac{|x-y|^2}{[XY]} &= \Delta_y |x-y|^2\frac{1}{[XY]} + 2\nabla_y |x-y|^2 \cdot \nabla_y \frac{1}{[XY]} + |x-y|^2\Delta_y \frac{1}{[XY]}.
\end{align*}
A direct computation yields 
$$
\nabla_y |x-y|^2 \cdot \nabla_y \frac{1}{[XY]} = \dfrac{-2|x|^2}{[XY]^3} (y - x)(y - x^*).
$$
Plugging all these we obtain part $(a)$ of the lemma.

For part $(b)$, one notices that for $x \in B_1,$
\[
\frac{\partial}{\partial \nu}\Delta_y G(x,y)\;= \; \nabla_{y} (\Delta_y G(x, y))\cdot y, \quad y \in \partial B_1.
\]
Then, we compute 
\begin{align*}
    \parder{}{y_i}\frac{1}{[XY]} = \frac{1}{[XY]^3}(x_i - y_i|x|^2),\quad 
    \parder{}{y_i}\frac{1}{|x-y|} = \frac{1}{|x-y|^3}(x_i - y_i),
\end{align*}
\begin{align*}
    \parder{}{y_i}\frac{1}{[XY]^3}\left((y-x)(y-x^*)\right) &= -\frac{3|x|^2}{[XY]^5}(y_i - x_i)(y-x)(y-x^*) \\
    &\quad + \frac{1}{[XY]^3}(2y_i - x_i - x_i).
\end{align*}
Thus, using the fact that then $[XY]=|x-y|$ when $y\in\partial B_1$, we obtain
\begin{align*}
    \nabla_y \Delta_y G(x, y) \cdot y &= C\left(-\frac{2|x|^4+6}{[XY]^3}(y-x^*)\cdot y + \frac{12|x|^4(y-x)(y-x^*)}{|x-y|^5}(y-x^*)\cdot y\right. \\
    &\left.+4(1-|x|^2)\frac{(y-x)\cdot y}{|x-y|^3} - \frac{4|x|^2(y-x^*)\cdot y}{[XY]^3}\right).
\end{align*}
\end{proof}

We are now in a position to derive the integral representation for positive 
solutions of \eqref{eq:biharmonic}. In other words, we are ready to prove 
Theorem~\ref{integral_rep}, which expresses any such solution in terms of the 
Green's function of the bi-Laplacian and highlights the underlying integral 
structure of the problem.

\begin{proof}
Using Green's identity for the bi-Laplace operator and the boundary conditions, we have
\begin{align*}
v(x) &= -\int_{\pb} \parder{v}{\nu}(y)\Delta_y G(x,y)d\sigma(y) + \int_{\pb}v(y)\parder{}{\nu}\Delta_y G(x,y)d\sigma(y) - \int_{B_1} G(x,y) f(v(y)) \, dy\\
& =  -b \underbrace{\int_{\pb}\Delta_y G(x,y)d\sigma(y)}_{I}\;+\; 
a \underbrace{\int_{\pb}\parder{}{\nu}\Delta_y G(x,y)d\sigma(y)}_{II} - \int_{B_1} G(x,y) f(v(y)) \, dy.
\end{align*}
{\bf Computation of $I_1$:} In order to compute \(I_1\), we first observe that for \(y \in \partial B_1\) we have \([XY] = |x - y|\). Let \(r = |x|\). Therefore, on \(\partial B_1\), we obtain
\begin{equation*}
    \Delta_yG(x,y)\;=\;C\left(\frac{2r^2+2}{|x-y|}-\frac{4r^2}{|x-y|^3}(y-x)(y-x^*)\right).
\end{equation*}

We first compute the first term of the integral in $I_1$. Rotate coordinates so that \(x=r(0,0,1)\). In spherical coordinates for \(y\in \partial B_1\),
\[
y=(\sin\theta\cos\phi,\;\sin\theta\sin\phi,\;\cos\theta),\qquad
\theta\in[0,\pi],\ \phi\in[0,2\pi),
\]
and
\[
|x-y|^2 = r^2+1-2r\cos\theta,\qquad d\sigma(y)=\sin\theta\,d\theta\,d\phi.
\]
Thus
\begin{align*}
    I_1^1 := \int_{\partial B_1} \frac{2r^2+2}{|x-y|} \;d\sigma(y) \;= \;(2r^2+2)\int_{0}^{2\pi}\int_{0}^{\pi}\frac{\sin\theta}{\sqrt{r^2+1-2r\cos\theta}}\, 
    d\theta d\phi. 
\end{align*}

Put \(u = \cos\theta\) so \(du = -\sin\theta\,d\theta\). The \(\theta\)-integral becomes
\[
\int_{-1}^{1}\frac{1}{\sqrt{1+r^2-2ru}}\,du = 2.
\]
Therefore we conclude 
\[
 I_1^1= (2r^2+2) 2\pi\int_{-1}^{1}\frac{1}{\sqrt{r^2+1-2rt}}\, dt=8\pi(r^2+1).
\]

Now we aim to compute the second term of the integral $I_1.$ Let \(x \in B_1\subset \mathbb{R}^3\) and set \(r = |x|\). Consider
\[
I_1^2 := \int_{\partial B_1} \frac{4r^2 (y-x)\cdot (y-x^*)}{|x-y|^3}\, d\sigma(y), \qquad y \in \partial B_1.
\]

Since \(|y|=1\), we compute
\[
(y-x)\cdot (y-x^*)
= |y|^2 - y\cdot x - y\cdot x^* + x\cdot x^*
= 1 - y\cdot x - \frac{1}{r^2} y\cdot x + 1
= 2 - \Big(1+\frac{1}{r^2}\Big) y\cdot x.
\]
Thus
\[
I_1^2 = \int_{\partial B_1} \frac{4r^2\big(2 - (1 + r^{-2})y\cdot x\big)}{|x-y|^3}\, d\sigma(y)
= 8r^2 \int_{\partial B_1} \frac{d\sigma(y)}{|x-y|^3}
 - 4(r^2+1) \int_{\partial B_1} \frac{y\cdot x}{|x-y|^3}\, d\sigma(y).
\]

Now recall that
\[
U(x) := \int_{\partial B_1} \frac{d\sigma(y)}{|x-y|}
= 4\pi, \qquad |x|<1.
\]
Differentiating with respect to the radial variable gives
\[
0 = \frac{d}{dr}U(x)
= \frac{x}{r}\cdot \int_{\partial B_1} \nabla_x \frac{1}{|x-y|}\, d\sigma(y)
= -\,\frac{1}{r} \int_{\partial B_1} \frac{r^2 - x\cdot y}{|x-y|^3}\, d\sigma(y).
\]
Hence
\[
\int_{\partial B_1} \frac{x\cdot y}{|x-y|^3}\, d\sigma(y)
= r^2 \int_{\partial B_1} \frac{d\sigma(y)}{|x-y|^3}.
\]

Therefore
\[
I_1^2 = \big(8r^2 - 4(r^2+1)r^2\big)
     \int_{\partial B_1} \frac{d\sigma(y)}{|x-y|^3}
= 4r^2(1-r^2) \int_{\partial B_1} \frac{d\sigma(y)}{|x-y|^3}.
\]

To evaluate
\[
A(r) := \int_{\partial B_1} \frac{d\sigma(y)}{|x-y|^3},
\]
align \(x\) with the north pole and note that \(|x-y|^2 = 1+r^2 - 2r\cos\theta\). Then
\[
A(r)
= 2\pi \int_{-1}^{1} \frac{du}{(1+r^2 - 2ru)^{3/2}}
= 2\pi \cdot \frac{2}{1-r^2}
= \frac{4\pi}{1-r^2}, \qquad 0 \le r < 1.
\]

Thus
\[
I_1^2 = 4r^2(1-r^2) \cdot \frac{4\pi}{1-r^2}
= 16\pi r^2.
\]

\[
\int_{\partial B_1} \frac{4r^2 (y-x)\cdot (y-x^*)}{|x-y|^3} \, d\sigma(y)
= 16\pi r^2,
\qquad x\in B_1,\ r = |x|.
\]
Finally,
\[
I_1 \;=\; \int_{\partial B_1} \Delta_y G(x,y) \, d\sigma(y) = C(I_1^1 - I_1^2) = 8\pi C (1-r^2).
\]
\medskip 
{\bf Computation of $II$:} We need to compute the integral involving the normal derivative of $\Delta_y G(x,y)$ on $\partial B_1$. Using Lemma~\ref{lemma-gr}, we have 
\begin{align*}
   &\int_{\partial B_1} \frac{\partial}{\partial \nu}\Delta_y G(x,y)\; d\sigma(y)\\
   &=C
    \left(\underbrace{\int_{\partial B_1}-\frac{2|x|^4+6}{[XY]^3}(y-x^*)\cdot y\; d\sigma(y)}_{{II}_1} 
   + \underbrace{\int_{\partial B_1}\frac{12|x|^4(y-x)(y-x^*)}{|x-y|^5}(y-x^*)\cdot y\; d\sigma(y)}_{{II}_2} \right. \\
    & \left.+4(1-|x|^2)\underbrace{\int_{\partial B_1}\frac{(y-x)\cdot y}{|x-y|^3}\; d\sigma(y)}_{{II}_3} - 4|x|^2\underbrace{\int_{\partial B_1}\frac{(y-x^*)\cdot y}{|x-y|^3}\; d\sigma(y)}_{{II}_4}\right).
\end{align*}
Let us first evaluate ${II}_4.$ It is easy to see
\[
(y - x^*) \cdot y = |y|^2 - y \cdot x^* = 1 - y \cdot x^*,
\]
since \(|y| = 1\) on \(\partial B_1\).  Then
\[
\int_{\partial B_1} \frac{(y - x^*) \cdot y}{|x-y|^3} \, d\sigma(y)
= \int_{\partial B_1} \frac{d\sigma(y)}{|x-y|^3} - \int_{\partial B_1} \frac{y \cdot x^*}{|x-y|^3} \, d\sigma(y).
\]
As before we have 
\[ 
\int_{\partial B_1} \dfrac{x \cdot y}{|x -y|^3} \; d \sigma(y) \; =\; |x|^2 \int_{\partial B_1} \dfrac{d \sigma(y)}{|x -y|^3}.
\]
Thus,
\[
\int_{\partial B_1} \frac{y \cdot x^*}{|x-y|^3} \, d\sigma(y)
= \frac{1}{r^2} \int_{\partial B_1} \frac{x \cdot y}{|x-y|^3} \, d\sigma(y)
= \frac{1}{r^2} \cdot r^2 \int_{\partial B_1} \frac{d\sigma(y)}{|x-y|^3} 
= \int_{\partial B_1} \frac{d\sigma(y)}{|x-y|^3}.
\]

Hence we conclude ${II}_4 =0.$ A similar computation shows that ${II}_1 =0.$

We want to compute for $x \neq 0,$
\[
{II}_2 \;=\; 12r^2 \int_{\partial B_1} 
\frac{(y-x)\cdot \bigl(y - \frac{x}{|x|^2}\bigr)\,\bigl(y - \frac{x}{|x|^2}\bigr)\cdot y}
{|x-y|^5}\; d\sigma(y), \quad x \in B_1.
\]
Then we can further simplify as follows. Denote
\[
|x-y|^2 \;=\; 1+r^2 - 2rt \;=:\; A(t), 
\qquad t=\cos\theta.
\]
The integrand becomes
\[
 {II}_2 \;=\;6\pi \left[
 2r^2 \int_{-1}^1 \frac{1}{A(t)^{5/2}}dt
 - \Bigl(r^3 + 3r\Bigr) \int_{-1}^1 \frac{t}{A(t)^{5/2}}\, dt
 + \Bigl(r^2+1\Bigr) \int_{-1}^1 \frac{t^2}{A(t)^{5/2}}\, dt
\right].
\]
Direct computation shows that
\begin{align*}
    \int_{-1}^1 \frac{1}{A^{5/2}}dt=\frac{2(3+r^2)}{3(1-r^2)^3}; \quad
    \int_{-1}^1 \frac{t}{A^{5/2}}dt=\frac{2r(5-r^2)}{3(1-r^2)^3}; \quad 
    \int_{-1}^1 \frac{t^2}{A^{5/2}}dt=\frac{2(-2r^4+5r^2+1)}{3(1-r^2)^3}.
\end{align*}
Therefore, plugging all these terms and simplifying, we obtain ${II}_2 = 16 \pi|x|^2$. A similar trick applies to compute ${II}_3$ and we get ${II}_3 = 4\pi$. Hence 
\[\int_{\partial B_1}\frac{\partial}{\partial \nu}\Delta_y G(x,y)d\sigma(y) = 16\pi,\]
and the proof is complete.
\end{proof}

\section{ Radial symmetry, and Proof of Theorem~\ref{thm:main} using the moving plane method}
This section is devoted to establishing the radial symmetry of positive 
solutions to equation~\eqref{eq:biharmonic}. Our approach is based on the 
method of moving planes. By means of the integral representation formula 
obtained in Theorem~\ref{integral_rep}, we convert \eqref{eq:biharmonic} into 
the equivalent integral equation \eqref{eq:integral_rep}. We then apply the 
moving plane argument directly to the integral formulation, which enables us 
to overcome the lack of a maximum principle for the fourth-order operator and 
ultimately prove that any positive solution must be radially symmetric and 
strictly \emph{increasing} with respect to the radial variable. 
\subsection{The moving plane argument} Before proceeding with the moving plane method, we begin by establishing the necessary notations.

In what follows, for each $\lambda \in [0,1]$, we introduce the hyperplane
\[
T_\lambda := \{x \in \mathbb{R}^n : x_1 = \lambda\},
\]
and the associated region
\[
\Sigma_\lambda := \{x \in B_1 : x_1 < \lambda\},
\]
which corresponds to the part of the unit ball lying to the left of $T_\lambda$.

For any point $x \in \mathbb{R}^n$, we denote by $\bar{x}$ its reflection with respect to the hyperplane $T_\lambda$. Furthermore, we extend the Green's function $G$ trivially to 
\[
\mathbb{R}^n \times \mathbb{R}^n \setminus \{(x,x) : x \in \mathbb{R}^n\}
\]
by setting
\[
G(x,y) = 0 \quad \text{whenever } |x| \ge 1 \text{ or } |y| \ge 1.
\]

The following lemmas, namely Lemma \ref{G_sign}--\ref{G_compare}, are classical results and can be found in the literature; see, for example, \cites{GGS10, LLW}.
\begin{lemma}\label{G_sign}
    Let $\lambda\in [0,1)$, then for every $x\in B_1\cap T_\lambda$ and $y\in\Sigma_\lambda$, we have
    \begin{align*}
        &G_{x_1}(x, y)<0\\
        &G_{ x_1}(x, y)+G_{x_1}(x,\bar{y})\leq 0.
    \end{align*}
    Moreover, the second inequality is strict if $\lambda>0$.
\end{lemma}

\medskip
\begin{lemma}\label{G_compare}
    Let $\lambda \in (0, 1)$. For all $x, y \in \Sigma_\lambda, x \neq y$, we have
    \begin{align*}
        &G(x, y)>\max\{G(x,\bar{y}),G(\bar{x},y)\}\\
        &G(x, y)-G(\bar{x},\bar{y})> |G(x,\bar{y})-G(\bar{x},y)|.
    \end{align*}
\end{lemma}

\begin{proposition}
Under the assumptions of Theorem \ref{thm:main}, $u$ must be radially symmetric with respect to the 
origin, and \(u\) is strictly decreasing in \(r = |x|\).
\end{proposition}

The proof of the proposition follows three standard steps of the moving plane method.

\textbf{Step 1:} We first show that the moving plane argument can be initiated near the boundary point through the following lemma.
\begin{lemma}
    Let $0<\lambda<1$, and suppose that $v(x)\leq v(\bar{x})$ for all $x\in\Sigma_\lambda$. Then $\frac{\partial v}{\partial x_1}>0$ on $T_l\cap B$ for all $l\in(\lambda-\gamma,\lambda)$.
\end{lemma}
\begin{proof}
    For all $x\in T_l\cap B_1$ there holds
    \begin{align*}
        \frac{\partial v}{\partial x_1}(x) 
        &= \frac{3\sqrt{\pi}}{2}b\int_{\pb} \parder{}{x_1}(|x|^2-1)\,d\sigma(y) - \int_{\Sigma_\lambda} G_{x_1}(x,y) f(v(y)) \, dy
    \end{align*}
    By reflection, we have
    \begin{align*}
       -\int_{B_1} G_{x_1}(x,y) f(v(y)) \, dy &= -\int_{\Sigma_\lambda}G_{x_1}(x,y)f(v(y))+G_{x_1}(x,\bar{y})f(v(\bar{y}))\,dy\\
        &\geq -\int_{\Sigma_\lambda} \left[G_{x_1}(x,y)+G_{x_1}(x,\bar{y})\right] f(v(\bar{y})) \, dy > 0.
    \end{align*}
    Thus, $\frac{\partial v}{\partial x_1}(x)>0$ for all $x\in T_l\cap B_1$, whenever $b \geq 0$. Using the given hypothesis, there is some neighborhood $\mathcal{U}(x_0)$ of $x_0$ such that 
    \[\parder{v}{x_1}>0, \text{ for all } x\in \bigcup_{x_0\in T_\lambda\cap \partial B_1} \mathcal{U}(x_0)\cap B_1.\]
    Since $v(x_0)=a >0$ and $\parder{}{x_1}v(x,0)=b >0$, we then conclude that $\frac{\partial v}{\partial x_1}(x)>0$ for all $x\in T_l\cap B_1$ when $l$ is sufficiently close to $\lambda$.   
\end{proof}

\textbf{Step 2:}  
Now we start the moving plane procedure by shifting the plane $T_\lambda$ from the initial tangential position $T_1$ towards the interior of $B$. Recall the integral representation of $u$:
\begin{align}
v(x)
&= \int_{\partial B_1} \Bigl( v\,\partial_\nu (\Delta G) - (\partial_\nu v)\,\Delta G \Bigr) d\sigma_y 
- \int_{B_1} G(x,y)f(v(y))\,dy \nonumber\\
&= C_1 + C_2(|x|^2-1) - \int_{B_1} G(x,y)f(v(y))\,dy.
\end{align}

Next, we shall establish the following result: 
\begin{lemma}\label{initial_compare}
    There exists $\varepsilon>0$ such that for all $\lambda\in[1-\varepsilon,1)$ we have
    \begin{equation}\label{eq9}
    v(x)<v(\bar{x}) \text{ for } x\in\Sigma_\lambda, \;\;
    \frac{\partial v}{\partial x_1}(x)>0 \;\text{ for } x\in T_\lambda\cap B_1.
    \end{equation}
\end{lemma}
\begin{proof}
    Since $T_1\cap\,\partial B_1=\{e_1\}$, where $e_1=(1,0,\cdots,0)$. A similar argument concludes that there exists $\varepsilon>0$ such that $\frac{\partial v}{\partial x_1}(x)>0$, for
    $x\in B_1\setminus\Sigma_{1-2\varepsilon}$. As a result, (\ref{eq9}) holds for all $\lambda\in[1-\varepsilon, 1)$.
\end{proof}

\textbf{Step 3:} Afterwards, we continue to move the plane inward until it reaches the origin.
Let $$\Lambda:=\{\lambda\in (0, 1)\mid v(x)<v(\bar{x}), \,\forall x\in\Sigma_\lambda,\text{ and } \frac{\partial v}{\partial x_1}(x)>0, \forall x\in T_\lambda\cap B_1\}.$$
By Lemma \ref{initial_compare}, we know that
$[1-\varepsilon, 1)\subset \Lambda$. Let $\bar{\lambda}\in [0, 1)$ be the smallest number such that $(\bar{\lambda}, 1)\subset\Lambda$. We would like to show $\bar{\lambda}=0$ so that $\Lambda=(0, 1)$. Assume by contradiction that $\bar{\lambda}>0$, then there exists an
\begin{equation*}
    \gamma\in(0,\bar{\lambda}) \text{ such that } \frac{\partial v}{\partial x_1}<0, \text{ on } T_l\cap B_1,\; \forall\, l\in(\bar{\lambda}-\gamma,\bar{\lambda}).
\end{equation*}
Now for any $x\in\Sigma_{\bar{\lambda}}$,
\begin{align*}
    &v(\bar{x})-v(x)=C_2(|\bar{x}|^2-|x|^2)-\int_{B_1}(G(\bar{x},y)-G(x,y))f(v(y))dy\\
    &>-\int_{\Sigma_{\bar{\lambda}}} G(\bar{x},y)f(v(y)))dy-\int_{\Sigma^c_{\bar{\lambda}}} G(\bar{x},y)f(v(y))dy+\int_{\Sigma_{\bar{\lambda}}} G(x,y)f(v(y))dy+\int_{\Sigma^c_{\bar{\lambda}}} G(x,y)f(v(y))dy\\
    &=\int_{\Sigma_{\bar{\lambda}}}(G(x,y)-G(\bar{x},y))f(v(y))dy+\int_{\Sigma_{\bar{\lambda}}}(G(x,\bar{y})-G(\bar{x},\bar{y}))\tilde{f}(v(\bar{y}))dy,
\end{align*}
where $\tilde{f}$ denotes the trivial extension of $f$ to $\rn$.

Since $f(v(y))>\tilde{f}(v(\bar{y}))$, we have
$$v(\bar{x})-v(x)>\int_{\Sigma_{\bar{\lambda}}}(G(x,y)-G(\bar{x},y)+G(x,\bar{y})-G(\bar{x},\bar{y}))\tilde{f}(v(\bar{y}))dy\geq 0,$$
where the last inequality follows from Lemma \ref{G_compare}.
Thus, $v(x)>v(\bar{x})$ for all $x\in\Sigma_{\bar{\lambda}}$. Then by compactness, there exists $0<\gamma_1<\gamma$ such that $v(x)>v(\bar{x})$
for all $l\in(\bar{\lambda}-\gamma_1,\bar{\lambda}]$. which contradicts the choices of $\bar{\lambda}$.
Therefore, $\Lambda=(0, 1)$, and $v(\bar{x}_1,x_2,\cdots,x_n)\geq v(x_1,x_2,\cdots,x_n)$ for any $x_1\geq 0$. Finally, we conclude that $u$ is radially symmetric, since the $x_1$-direction can be chosen arbitrarily.

\section{Remarks on the Integral equation on the hyperbolic space}\label{apx}

In this section, we focus on the integral formulation associated with equation~\eqref{eq:Paneitz}. Our goal is to demonstrate that, in contrast to the Euclidean case, the differential equation and its corresponding integral equation are not equivalent in the hyperbolic space. This discrepancy highlights a new phenomenon arising from the geometry of the hyperbolic space, which does not appear in the classical Euclidean setting.

Before proceeding further, we establish several key properties of the Green's function on the hyperbolic space corresponding to the Paneitz operator $P_2$. These preliminary results will be essential in analyzing the structure of the integral representation and in understanding the subtle differences between the integral and differential formulations.

\subsection{Asymptotic behaviors of the Green's function of \texorpdfstring{$P_2$}{} on \texorpdfstring{$\mathbb{H}^3$}{}}
Before proceeding further, let us first recall the following lemma, which will play an essential role in our analysis.

\begin{lemma}\cite[sLemma~3.3]{LY22}
Let $n \geq 3$ and $\nu \geq 0$. Then
\begin{align*}
&\left(\nu^2 - \frac{(n-1)^2}{4} - \Delta_{\mathbb{H}}\right)^{-1}\\
&= \frac{\Gamma\!\left(\frac{n-1}{2} + \nu\right)\Gamma\!\left(\nu + \frac{1}{2}\right)}{4\pi^{\frac{n}{2}}\Gamma(2\nu + 1)}
\, \frac{\left(\cosh \frac{\rho}{2}\right)^{\,n-3-2\nu}}{(\sinh \rho)^{\,n-2}}
\, F\!\left(\nu - \frac{n+3}{2},\, \nu + \frac{1}{2};\, 2\nu + 1;\, \cosh^{-2}\frac{\rho}{2}\right) \\
&= \frac{\Gamma\!\left(\frac{n-1}{2} + \nu\right)\Gamma\!\left(\nu + \frac{1}{2}\right)}{2^{\,n}\pi^{\frac{n}{2}}\Gamma(2\nu + 1)}
\, \left(\cosh \frac{\rho}{2}\right)^{1-n-2\nu}
\, F\!\left(\nu + \frac{n-1}{2},\, \nu + \frac{1}{2};\, 2\nu + 1;\, \cosh^{-2}\frac{\rho}{2}\right),
\end{align*}
where $\rho$ denotes the hyperbolic distance and $F$ is the Gauss hypergeometric function.
\end{lemma}

The above lemma led us to the following proposition. 

\begin{proposition} Let $P_2:= P_1(P_1+ 2),$
where $P_1= -\Delta_{\mathbb{H}^3} - \frac{3}{4}.$ Then There holds: 
\begin{align}\label{P_2_inverse}
    &G(\rho):=P_2^{-1}=\frac{1}{16\pi}\frac{(\cosh\frac{\rho}{2})^{-4}}{(\sinh\frac{\rho}{2})^{-1}}F\left(\frac{3}{2},2,3;\cosh^{-2}\frac{\rho}{2}\right),
\end{align}
where $\rho:=\rho(x, y).$ Moreover the following asymptotics hold: 
\begin{align*}
 G(x,y) &\sim A_0 + A_2\,\rho^{2} + O(\rho^{3}), && \rho\to 0,\\
 G(x,y) &\sim C\,e^{-\frac{3}{2}\rho}, && \rho\to \infty.
\end{align*}
\end{proposition}

\begin{proof}

The explicit formula for $G$ in \eqref{P_2_inverse} is obtained by substituting $n=3$ and $\nu = \tfrac{1}{2}$ into the above lemma.
Moreover, we have 
\begin{align*}
    F\left(\frac{3}{2},2,3;\cosh^{-2}\frac{\rho}{2}\right)=2\int_0^1t(1-t\cosh^{-2}\frac{\rho}{2})^{-\frac{3}{2}}dt.
\end{align*}

Let $a = \cosh^{-2}\!\left(\frac{\rho}{2}\right)$ and consider
\[
I = \int_0^1 t(1 - at)^{-\frac{3}{2}} \, dt,
\]
where \(0 < a < 1\). Setting \(u = 1 - at\), we compute
\begin{align*}
    I &= \int_1^{1-a} \frac{(1-u)}{a} \cdot u^{-\frac{3}{2}} \cdot \left(-\frac{du}{a}\right) 
      = \frac{1}{a^2} \int_{1-a}^1 (1-u) u^{-\frac{3}{2}}\,du 
       = \frac{1}{a^2} \int_{1-a}^1 \left(u^{-\frac{3}{2}} - u^{-\frac{1}{2}}\right)\,du \\
      &= \frac{2}{a^2}\left((1-a)^{-\frac{1}{2}} + (1-a)^{\frac{1}{2}} - 2\right).
\end{align*}

Since \(1-a = \tanh^2\!\left(\frac{\rho}{2}\right)\), we obtain for \(0 < \rho < \infty\),
\[
G(\rho)=\frac{1}{8\pi}\sinh\!\left(\frac{\rho}{2}\right)\left(\tanh\!\left(\frac{\rho}{2}\right)+\coth\!\left(\frac{\rho}{2}\right)-2\right) > 0.
\]

\underline{Case: \(a \to 0\) (i.e., \(\rho \to \infty\)).}  
We expand
\[
t(1-at)^{-\frac{3}{2}} \approx t\left(1 + \frac{3}{2}at\right) = t + \frac{3}{2}at^2,
\]
so that \(I \to \frac{1}{2}\). Hence,
\[
P_2^{-1}(\rho) \sim \frac{8}{16\pi}\cosh^{-3}\!\left(\frac{\rho}{2}\right) 
    \sim \frac{1}{2\pi} e^{-\frac{3}{2}\rho}, \qquad \rho \to \infty.
\]

\underline{Case: \(a \to 1\) (i.e., \(\rho \to 0\)).}
Here \(I \to \frac{4}{\rho}\), and using the approximations
\[
\cosh\!\left(\frac{\rho}{2}\right) \sim 1 + \frac{\rho^2}{8}, 
\qquad \sinh \rho \sim \frac{\rho}{2},
\]
we obtain
\[
P_2^{-1}(\rho) \sim \frac{1}{4\pi}, \qquad \rho \to 0.
\]
\end{proof}
Next, we establish the monotonicity property of the Green's function $G$ associated with the Paneitz operator $P_2$ on the hyperbolic space $\mathbb{H}^3$.

\begin{proposition}
The Green's function $G(x,y)$ associated with the Paneitz operator $P_2$ on $\mathbb{H}^3$ is strictly positive and monotonically decreasing with respect to the geodesic distance $\rho(x,y)$.
\end{proposition}
\begin{proof}
We now simplify the hypergeometric function
\[
F\left(\frac{3}{2}, 2, 3; z\right) = 2 \int_0^1 t(1 - tz)^{-\frac{3}{2}}\, dt,
\]
where \(z = \cosh^{-2}\!\left(\frac{\rho}{2}\right)\). Substituting this expression for \(z\), we obtain
\[
F\left(\frac{3}{2}, 2, 3; \cosh^{-2}\!\left(\frac{\rho}{2}\right)\right)
   = 2\int_0^1 t\left(1 - t\cosh^{-2}\!\left(\frac{\rho}{2}\right)\right)^{-\frac{3}{2}} dt.
\]

Letting \(u=t\) and \(v = 1 - u\cosh^{-2}\!\left(\frac{\rho}{2}\right)\), the integral becomes
\begin{align*}
F &= 2 \int_{1}^{\tanh^2(\rho/2)} (1-v)\,v^{-3/2} \cdot \bigl(-\cosh^2(\rho/2)\bigr)\, dv 
  = 2 \cosh^4\!\left(\frac{\rho}{2}\right) \int_{\tanh^2(\rho/2)}^{1} (1-v)\,v^{-3/2} \, dv \\
  &= \cosh^4\!\left(\frac{\rho}{2}\right)
     \left[-2v^{-1/2} - 2v^{1/2}\right]_{\tanh^2(\rho/2)}^{1} 
  = \cosh^4\!\left(\frac{\rho}{2}\right)\left(-4 + 2\,\coth\!\left(\frac{\rho}{2}\right) + 2\,\tanh\!\left(\frac{\rho}{2}\right)\right).
\end{align*}

Using the identity
\[
\coth\!\left(\frac{\rho}{2}\right) + \tanh\!\left(\frac{\rho}{2}\right)
= \frac{\cosh^2(\rho/2) + \sinh^2(\rho/2)}
       {\sinh(\rho/2)\cosh(\rho/2)}
= 2\,\coth(\rho),
\]
we obtain a further simplification
\[
F = 2\cosh^4\!\left(\frac{\rho}{2}\right)\left(-4 + 4\coth(\rho)\right)
= \frac{2\cosh^3(\rho/2)}{e^{\rho}\sinh(\rho/2)}.
\]

Thus, the Green's function becomes
\[
P_2^{-1}(\rho)
= \frac{1}{16\pi} \cdot 
  \frac{\sinh(\rho/2)}{\cosh^4(\rho/2)} \cdot
  \frac{2\cosh^3(\rho/2)}{e^{\rho}\sinh(\rho/2)}
= \frac{1}{8\pi\, e^{\rho}\cosh(\rho/2)}.
\]

We now compute its radial derivative:
\begin{align*}
\frac{d}{d\rho}P_2^{-1}(\rho)
&= \frac{1}{8\pi} \frac{d}{d\rho}\left(e^{-\rho}\sech\!\left(\frac{\rho}{2}\right)\right) \\
&= \frac{1}{8\pi} \left(-e^{-\rho}\sech\!\left(\frac{\rho}{2}\right)
     \left(1 + \frac{1}{2}\tanh\!\left(\frac{\rho}{2}\right)\right)\right) < 0,
\end{align*}
since \(e^{-\rho}\), \(\sech(\rho/2)\), and \(\tanh(\rho/2)\) are strictly positive for all \(\rho>0\).

Hence, \(P_2^{-1}(\rho)\) is strictly positive and monotonically decreasing on \((0,\infty)\).

\end{proof}

\subsection{Nonexistence of solutions}

\begin{proposition}
    There is no solution $u$ which has exponential growth at infinity to the integral equation \eqref{eq:green-IE}.
\end{proposition}

\begin{proof}
    
From the asymptotics of $G$, one has (as the hyperbolic distance $d(x,y)=\rho\to 0$ and $\rho\to\infty$)
\begin{align*}
 G(x,y) &\sim A_0 + A_2\,\rho^{2} + O(\rho^{3}), && \rho\to 0,\\
 G(x,y) &\sim C\,e^{-\frac{3}{2}\rho}, && \rho\to \infty,
\end{align*}
for some constants $A_0,A_2,C>0$. 
Assume by contradiction that there exists a solution $u$, then we have
\begin{equation}
     \int_{\mathbb{H}^n} \left|G(x,y)u^{-7}(y)\right|dV_y \lesssim \int_0^\infty e^{-\frac{3}{2}\rho(x,y)} e^{-\frac{7}{2}\rho(y,0)} \sinh^2(\rho(y,0))\,d\rho(y,0) <\infty.
\end{equation}
However, the dominated convergence theorem would force
\begin{equation}
 \lim_{x\to\infty} |u(x)| \leq \int_{\mathbb{H}^n} \lim_{x\to\infty}G(x,y)u^{-7}(y) dV_y = 0,
\end{equation}
contradicting the growth assumption of $u(x)\sim e^{\rho/2}$.
\end{proof}

%%%%%%%%%%%

\section*{Acknowledgments} 
  D. Ganguly was partially supported by the
SERB MATRICS (MTR/2023/000331). J. Li is supported by the National Natural Science Foundation of China (No.12571127). G. Lu was partially supported by a grant from the Simons Foundation.

\end{document}